\documentclass[12pt,english]{amsart}

\usepackage[T1]{fontenc}
\usepackage{amssymb, setspace}
\usepackage{paralist}
\usepackage{babel}
\usepackage[pdftex]{graphicx}    
\usepackage[leqno]{amsmath}
\usepackage{amsthm}
\usepackage[pagebackref,colorlinks,linkcolor=red,citecolor=blue,urlcolor=blue,hypertexnames=true]{hyperref}
\usepackage{amsrefs}
\usepackage[matrix, arrow]{xy}
\usepackage[active]{srcltx}
\usepackage{nicefrac}
\usepackage{tikz}

\renewcommand{\deg}{{\rm deg}}

\newcommand{\R}{\mathbb R}

\newcommand{\C}{\mathbb C}

\newcommand{\trace}{{\rm tr}}
\newcommand{\cH}{\mathcal{H}}
\newcommand{\cM}{\mathcal{M}}

\newcommand{\cQ}{\mathcal{Q}}
\newcommand{\cL}{\mathcal{L}}

\newcommand{\phom}{P}
\newcommand{\xvar}{x}

\theoremstyle{plain}
\newtheorem*{theorem*}{Theorem}
\newtheorem{theorem}{Theorem}[section]
\newtheorem{corollary}[theorem]{Corollary}
\newtheorem{lemma}[theorem]{Lemma}
\newtheorem{proposition}[theorem]{Proposition}

\theoremstyle{definition}
\newtheorem*{definition*}{Definition}
\newtheorem{definition}[theorem]{Definition}

\newtheorem{setup}[theorem]{Setup}

\theoremstyle{remark}

\newtheorem{remark}[theorem]{Remark}
\newtheorem{remarks}[theorem]{Remarks}
\newtheorem{example}[theorem]{Example}

\begin{document}

\onehalfspace

\title{Determinantal Representations and the Hermite Matrix}

\author{Tim Netzer}
\address{Tim Netzer, Universit\"at Leipzig, Germany}
\email{netzer@math.uni-leipzig.de}

\author{Daniel Plaumann}
\address{Daniel Plaumann, Universit\"at Konstanz, Germany}
\email{Daniel.Plaumann@uni-konstanz.de}

\author{Andreas Thom}
\address{Andreas Thom, Universit\"at Leipzig, Germany}
\email{thom@math.uni-leipzig.de}

\begin{abstract} We consider the problem of writing real polynomials
  as determinants of symmetric linear matrix polynomials. This problem
  of algebraic geometry, whose roots go back to the nineteenth
  century, has recently received new attention from the viewpoint of
  convex optimization. We relate the question to sums of squares
  decompositions of a certain Hermite matrix. If some power of a
  polynomial admits a definite determinantal representation, then its
  Hermite matrix is a sum of squares.  Conversely, we show how a
  determinantal representation can sometimes be constructed from a
  sums-of-squares decomposition of the Hermite matrix. We finally show
  that definite determinantal representations always exist, if one
  allows for denominators.\end{abstract}

\thanks{Travel for this project was partially supported by the
  Forschungsinitiative \emph{Real Algebraic Geometry and Emerging
    Applications} at the University of Konstanz. Daniel Plaumann
  gratefully acknowledges support through a Feodor Lynen return
  fellowship from the Alexander von Humboldt Foundation.}

\date{\today}

\maketitle


\section*{Introduction}

A polynomial $p\in\R[x]$ in $n$ variables $x=(x_1,\dots,x_n)$ with
$p(0)=1$ is
called a \emph{real-zero polynomial} if $p$ has only
real zeros along every line through the origin. The terms
\emph{hyperbolic} or \emph{real stable} polynomial are also common and
mean essentially the same, but usually for homogeneous polynomials. The
typical example is a polynomial given by a \emph{definite} (linear
  symmetric) \emph{determinantal representation}
\[
p=\det(I + A_1x_1+\cdots+A_nx_n),
\]
where $A_1,\dots,A_n$ are real symmetric matrices and $I$ is the
identity. A representation of this form is a certificate for being a
real-zero polynomial. In other words, the fact that $p$ is a real-zero
polynomial is apparent from the representation. A definite
determinantal representation also provides a description of the
\emph{rigidly convex region} of $p$. This is the closed connected component
of the origin in the complement of the zero-set of $p$. It is always
convex, and given a definite determinantal representation of $p$, it
coincides with the set of points where the matrix polynomial
$I+A_1x_1+\cdots+A_nx_n$ is positive semidefinite. 

In recent years, real-zero polynomials and their determinantal
representations have been studied mostly with a view towards convex
optimization, specifically semidefinite and hyperbolic programming. In
general, one would like to answer the following questions:
\begin{enumerate}
\item Under what conditions does a real-zero polynomial have a
  definite determinantal representation?
\item If such a representation exists, what is the minimal matrix
  dimension and how can the representation be computed effectively?
\item If no such representation exists, what other certificates for
  being a real-zero polynomial are available?
\end{enumerate}
Question (1) is the most immediate and has consequently received the
most attention. It ties in with the theory of determinantal
hypersurfaces in complex algebraic geometry, whose roots go back to
the nineteenth century. Arguably the most important modern results are
the Helton-Vinnikov theorem in \cite{hevi}, which gives a positive
answer for $n=2$, and Br\"and\'en's negative results in higher
dimensions in \cite{bran}. Since there are various subtle variations
of the question, it is not always easy to figure out what is known and
what is not; we give a very brief overview after the introduction below.

Question (2), which should be of interest for practical
purposes, has not been studied very systematically so far. Even in the
case $n=2$, the classical
approach of Dixon for constructing determinantal
representations is quite algorithmic in nature but hard to carry out
in practice (see \cite{dixon}, or \cite{vinn} for a more modern presentation.)

One approach to Question (3) is to study the determinantal
representability of a suitable power or multiple of $p$ if no
representation for $p$ exists. This is motivated by the Generalized
Lax Conjecture, as described below. On the other hand, the real-zero
property does not have to be expressed by a determinantal
representation.  That a polynomial $p$ in one variable has only real
roots is equivalent to its \emph{Hermite matrix} being positive
semidefinite. This is a symmetric real matrix associated with $p$,
which provides one of the classical methods for root counting. To
treat the multivariate case, we use a parametrized version of the
Hermite matrix with polynomial entries. In a typical
sums-of-squares-relaxation approach common in polynomial optimization,
we then ask for the parametrized Hermite matrix $\cH(p)$ to be a sum
of squares, which means that there exists a matrix $\cQ$ such that
$\cH(p)=\cQ^T\cQ$.  (This is called a sum of squares rather than a
square, because $\cQ$ is allowed to be rectangular of any size). This
approach has been used before by Henrion in \cite{hen} and by Parrilo
(unpublished) as a relaxation for the real-zero property, which is exact in the two-dimensional case.

While the Hermite matrix provides a practical way of certifying the real-zero
property, having a definite determinantal representation of $p$ is
clearly much more desirable, since it also yields a description of
the rigidly convex region by a linear matrix inequality. And even if
one is only interested in the real-zero property, the multivariate
Hermite matrix is a fairly unwieldy object compared to the original
polynomial, and a sum-of-squares decomposition even more so.

Our main goal is therefore to use a sum-of-squares decomposition of
the parametrized Hermite matrix of a polynomial $p$ to construct, as
explicitly as possible, a definite determinantal representation of $p$,
or at least of some multiple of $p$. We first show in Section
\ref{sec:hermite-matrix} that a definite determinantal representation
of some power of $p$ of the correct size always yields a sum-of-squares
decomposition of $\cH(p)$ (Thm.~\ref{thm:sos}). In Section
\ref{section:construction}, we make an attempt at the converse. This
is partly motivated by our experimental finding that the Hermite
matrix of the V\'{a}mos polynomial, which is the counterexample of
Br\"and\'en, is not a sum of squares (Example
\ref{example:Vamos}). Note also that in the case $n=2$, where every
real-zero polynomial possesses a definite determinantal representation
by the Helton-Vinnikov theorem, the parametrized Hermite matrix can be reduced to the univariate case. It is therefore a sum of squares if and only if it is
positive semidefinite, by a result of Jakubovi{\v{c} \cite{jak}. Given a decomposition $\cH(p)=\cQ^T\cQ$, we
show that a definite determinantal representation of a multiple of $p$
can be found if a certain extension problem for linear maps on free
graded modules derived from $\cQ$ has a solution
(Thm.~\ref{thm:main}). Given $\cQ$, the search for such a solution
amounts only to solving a system of linear equations. This method can in principle also be applied if the sums of squares decomposition uses denominators.  Finally, we show
that by allowing a sum-of-squares decomposition with denominators,
which exists whenever $\cH(p)$ is positive semidefinite, one can
always obtain a determinantal representation with denominators:

\begin{theorem*}\ Let $p$ be a square-free real-zero polynomial with
  $p(0)=1$. There exists a symmetric matrix $\cM$ whose
  entries are real homogeneous rational functions of degree $1$ such
  that $p=\det(I+\cM)$.
\end{theorem*}

\noindent The precise statement is given in Thm.~\ref{thm:rat}.\\

\emph{Acknowledgements.} We would like to thank Didier Henrion, Pablo
Parrilo, Rainer Sinn, and Cynthia Vinzant for helpful comments and
discussions.

\section*{Known results}

\begin{itemize}
\item For $n=2$, every real-zero polynomial of degree $d$ has a real definite
  determinantal representation of matrix size $d$ by the Helton-Vinnikov
  theorem \cite{hevi}.
\item For $n\ge 3$ and $d$ sufficiently large, a simple count of
  parameters shows that only an exceptional set of polynomials can
  have a real determinantal representation of size $d$. The question
  whether every real-zero polynomial has a definite determinantal
  representation of \emph{any} size became known as the generalized Lax
  conjecture. 
\item The generalized Lax conjecture was disproven by Br\"and\'en who
  even showed the existence of real-zero polynomials $p$ such that no
  power $p^r$ has a determinantal representation of any size
  \cite{bran}. His smallest counterexample, the so-called
  \emph{V\'{a}mos polynomial}, is of degree $4$ in $8$
  variables (see \ref{example:Vamos} below).
\item Netzer and Thom \cite{neth} have proved that only an exceptional set of polynomials can have a determinantal representation,
even if one allows for matrices of arbitrary size. This is true for $n\geq 3$ and $d$ sufficiently large, or $d\geq 4 $ and $n$ sufficiently large.
They also show that if $p$ is a real-zero
  polynomial of degree $2$, then there exists $r\ge 1$ such that $p^r$
  has a determinantal representation. On the other hand, there exists such $p$ where one cannot
  take $r=1$.
\item Another result of Helton, McCullough and Vinnikov \cite{hmv} (see also
  Quarez \cite{qu}) says that every real polynomial has a real symmetric
  determinantal representation, though not necessarily a definite
  one. This means that the constant term in the matrix polynomial cannot be chosen to be the identity matrix in their result.
\item The most general form of the Lax conjecture says that every
  rigidly convex set is a spectrahedron. In terms of determinantal
  representations, this amounts to the following: For every real-zero
  polynomial $p$ there exists another real-zero polynomial $q$ such
  $pq$ has a real definite determinantal representation and such that
  $q$ is non-negative on the rigidly convex set of $p$. This conjecture
  is still wide open, even without the additional positivity condition
  on $q$. Note that if $pq$ has a definite determinantal
  representation, then $q$ is automatically a real-zero polynomial.
\end{itemize}

\section{The Hermite Matrix}
\label{sec:hermite-matrix}

In this section we introduce the parametrized Hermite matrix $\cH(p)$ of a polynomial. It is positive semidefinite at each point if and only if $p$ is a real zero polynomial. If some power of $p$ admits a determinantal representation of the correct size, then $\cH(p)$ even turns out to be a sum of squares of polynomial matrices. 

Let $p=t^d+p_1t^{d-1}+\cdots+p_{d-1}t+p_d\in\R[t]$ be a monic
univariate polynomial of degree $d$ and let
$\lambda_1,\ldots\lambda_d$ be the complex zeros of $p$. Then
\[
N_{k}(p)=\sum_{i=1}^d \lambda_i^k
\]
is called the $k$-th {\it Newton sum} of $p$. The Newton sums are
symmetric functions in the roots, and can thus be expressed as
polynomials in the coefficients $p_i$ of $p$.  The {\it Hermite
  matrix} of $p$ is the symmetric $d\times d$ matrix
\[
H(p):=\left( N_{i+j-2}(p)\right)_{i,j=1,\ldots d}.
\]
It is a Hankel matrix whose entries are polynomial expressions in the
coefficients of $p$. Note that $H(p)=V^TV$, where $V$ is the
Vandermonde matrix with coefficients $\lambda_1,\dots,\lambda_d$. 

The following well-known fact goes back to Hermite. For a proof, see for example Theorem 4.59 in Basu, Pollack and Roy
\cite{bpr}.

\begin{theorem}\label{thm:HermiteMatrix}
  Let $p\in\R[t]$ be a monic polynomial. The rank of $H(p)$ is equal to the number of distinct zeros of $p$ in $\C$. The signature of the Hermite
  matrix $H(p)$ is equal to the number of distinct real zeros of
  $p$. 
  
   In particular, $H(p)$ is positive definite if and only if all
  zeros of $p$ are real and distinct, and $H(p)$ is positive semidefinite if an only of all zeros are real. \qed
\end{theorem}

Now let $p\in\R[\xvar]$ be a polynomial of degree $d$ in $n$ variables
$\xvar=(x_1,\dots,x_n)$. The polynomial $p$ is called a \emph{real-zero
  polynomial} (with respect to the origin) if $p(0)=1$ and  for every $a\in\R^n$,
the univariate polynomial $p(ta)\in\R[t]$ has only real zeros. We want
to express this condition in terms of a Hermite
matrix. Write $p=\sum_{i=0}^d p_i$ with $p_i$
homogeneous of degree $i$, and let $\phom(\xvar,t)=\sum_{i=0}^d p_it^{d-i}$ be
the homogenization of $p$ with respect to an additional variable $t$.
We consider $\phom$ as a monic univariate polynomial in $t$ and call the
Hermite matrix $H(\phom)$ the \emph{parametrized Hermite matrix of $p$},
denoted $\cH(p)$. Its entries are polynomials in the homogeneous
parts $p_i$ of $p$. The $(i,j)$-entry is a homogeneous polynomial in $\xvar$ of degree $i+j-2.$

\begin{corollary}
  A polynomial $p\in\R[\xvar]$ with $p(0)=1$ is a real-zero polynomial if
  and only if the matrix $\cH(p)(a)$ is positive semidefinite for
  all $a\in\R^n$.
\end{corollary}

\begin{proof}
  By Theorem~\ref{thm:HermiteMatrix}, $\cH(p)(a)$ is positive
  semidefinite for $a\in\R^n$ if and only if the univariate polynomial
  $t^dp(a_1t^{-1},\dots,a_nt{-1})$ has only real zeros. Substituting $t^{-1}$
  for $t$, we see that this is equivalent to $p(ta)$ having only
  real zeros.
\end{proof}

The following is Proposition 2.1 in Netzer and Thom \cite{neth}.

\begin{proposition}\label{prop:eigenvalues} Let $\cM=x_1M_1+\cdots
  +x_nM_n$ be a symmetric linear matrix polynomial, and let
  $p=\det(I-\cM)$. Then for each $a\in\R^n,$ the nonzero eigenvalues
  of $\cM(a)$ are in one to one correspondence with the zeros of the
  univariate polynomial $p(ta)$, counting multiplicities. The
  correspondence is given by the rule $\lambda \mapsto
  \frac{1}{\lambda}$.\qed
\end{proposition}

\begin{lemma}\label{lemma:hermtrace}
  Let $p\in\R[\xvar]$ be a real-zero polynomial of degree $d$, and assume that
  $p^r=\det(I-\cM)$ is a symmetric determinantal representation of size $k$, for some $r>0$. Then
\[
\cH(p)_{i,j}=\frac 1r\cdot\biggl(\trace\bigl(\cM^{i+j-2}\bigr)\biggr),
\]
except possibly for $(i,j)=(1,1)$, where $\cH(p)_{1,1}=d$ and $\trace\bigl(\cM^0\bigr)=k$.
\end{lemma}

\begin{proof}
  For each $a\in \R^n$, the trace of $\cM(a)^{s}$ is the $s$-power
  sum of the nonzero eigenvalues of $\cM(a)$. These eigenvalues are
  the inverses of the zeros of $p(ta)$, by Proposition
  \ref{prop:eigenvalues}, but each such zero gives rise to $r$ many
  eigenvalues. Since the zeros of $p(ta)$ correspond to the inverses
  of the zeros of $t^dp(t^{-1}a)$, the trace of $\cM(a)^s$ equals the
  $s$-power sum of the zeros of $t^d p(t^{-1}a)$ multiplied with
  $r$. This proves the claim.
\end{proof}

\begin{definition}
Let $\cH\in{\rm Sym}_d(\R[\xvar])$ be a symmetric matrix with polynomial entries. $\cH$ is a \emph{sum of squares}, if there is a $d'\times d$-matrix $\cQ$ with polynomial entries, such that $\cH=\cQ^T\cQ$. This is equivalent to the existence of $d'$ many $d$-vectors $\cQ_i$ with polynomial entries, such that $\cH=\sum_{i=1}^k \cQ_i\cQ_i^T.$
\end{definition}

\begin{theorem}\label{thm:sos} 
  Let $p\in\R[\xvar]$ be a real-zero polynomial of degree $d$. If a power
  $p^r$ admits a definite determinantal representation of
  size $r\cdot d$, for some $r>0$, then the parametrized Hermite
  matrix $\cH(p)$ is a sum of squares.
\end{theorem}

\begin{proof} Let $p^r=\det(I-\cM)$ with $\cM$ of size $k=rd$, and
  denote by $q^{(s)}_{\ell m}$ the $(\ell,m)$-entry of $\cM^s$. Put 
$\cQ_{\ell m}=\left(q^{(0)}_{\ell m}, \ldots,q^{(d-1)}_{\ell m}\right)^T\in
\R[\xvar]^d$. Then we find 
\[
\sum_{\ell,m=1}^k \cQ_{\ell m}\cQ_{\ell m}^T = \left(\sum_{\ell,m=1}^k
  q^{(i-1)}_{\ell m}q^{(j-1)}_{\ell m} \right)_{i,j=1,\ldots, d}= \left( {\rm
    tr}(\cM^{i-1}\cM^{j-1})\right)_{i,j=1,\dots,d}=r \cH(p),
\]
by Lemma \ref{lemma:hermtrace}.
\end{proof}

\begin{remarks}(1) If the determinantal representation of $p^r$ is of size $k>rd$, then $\cH(p)$ becomes a sum of squares after increasing the $(1,1)$-entry from $d$ to $k/r$. This is clear from the above proof.

(2) It was shown in Netzer and Thom \cite{neth} that if a polynomial
$p$ admits a definite determinantal representation, then it admits one of size $dn$, where $d$ is the degree of $p$ and $n$ is the number of variables. So if \emph{any} power  $p^r$ admits a determinantal representation of \emph{any} size, then $\cH(p)$ is a sum of squares, after increasing the $(1,1)$-entry from $d$ to $dn$. Note that this is independent of $r.$

(3) The determinant of $\cH(p)$ is the discriminant of $t^dp(t^{-1}x)$
in $t$. If $\cH(p)=\cQ^T\cQ$, it follows from the Cauchy-Binet formula
that the determinant of $\cH(p)$ is a sum of squares in $\R[x]$.
Thus, by the above theorem, the discriminant of $\det(tI+\cM)$ in $t$
is a sum of squares, a fact that has long been known, at least since
Borchardt's work from 1846 \cite{borch}.

(4) The sums-of-squares decomposition of $\cH(p)$ obtained by
Thm.~\ref{thm:sos} from a determinantal representation
$p^r=\det(I-\cM)$ is extremely special. In principle, it is possible
to characterize the decompositions of $\cH(p)$ coming from a
determinantal representation by a recurrence relation that they must
satisfy. But this does not appear to be a promising approach for finding 
determinantal representations.

\end{remarks}

\begin{example}\label{quad1}
It was shown in Netzer and Thom \cite{neth} that if $p$ is quadratic,
a high enough power admits a definite determinantal
representation of the correct size. Thus $\cH(p)$ is a sum of squares in this
case. This can also be shown directly. Write 
\[
p=\xvar^TA\xvar + b^T\xvar +1
\]
with $A\in {\rm Sym}_n(\R)$ and $b\in\R^n.$ Then $p$ is a real-zero
polynomial if and only if $bb^T-4A\succeq 0$, as is easily checked. We
find  $t^2p(t^{-1}\xvar)=\xvar^T A\xvar + b^T\xvar \cdot t + t^2,$ and so we compute 
\[
\cH(p)=\left(\begin{array}{cc}2 & -b^T\xvar \\-b^T\xvar & \xvar^T (
    bb^T-2A)\xvar\end{array}\right).
\]
Write $bb^T-4A=\sum_{i=1}^n v_iv_i^T$ as a sum of squares of column
vectors $v_i\in\R^n$. Set 
\[
\cQ=\left(\begin{array}{cc}1 & -\frac12 b^T \xvar \\0 & \frac12 v_1^T \xvar
    \\\vdots & \vdots \\0 & \frac12 v_n^T \xvar\end{array}\right).
\]
Then $\cH(p)=2\cdot \cQ^T\cQ$.\qed
\end{example}

\begin{example}\label{example:Vamos}
We consider Br\"and\'{e}n's example from \cite{bran}. It is  constructed from the V\'{a}mos cube as shown in   Figure \ref{branden}. Its  set of bases $\mathcal{B}$ consists of all four element subsets of $\{1,\ldots,8\}$ that do not lie in one of the five affine hyperplanes. Define $$q:= \sum_{B\in\mathcal{B}}\ \prod_{i\in B} x_i,$$ a degree four polynomial in $\R[x_1,\ldots,x_8].$ It contains as its terms the product of any choice of  four pairwisely different variables, except for the following five: $$ x_1x_4x_5x_6, x_2x_3x_5x_6, x_2x_3x_7x_8, x_1x_4x_7x_8, x_1x_2x_3x_4.$$  Now $p=q(x_1+1,\ldots,x_8+1)$ turns out to be a real-zero polynomial, of which Br\"and\'{e}n has shown that no power has a determinantal representation. 

\begin{figure}[ht]\caption{The V\'{a}mos Cube}\label{branden}
\begin{center} 
\begin{tikzpicture}[]
\draw[thick] (0,0) -- (1.5,-0.5){};
\draw[thick] (1.5,-0.5) -- (5,1){};
\draw[thick] (0,0) -- (0.8,-1.2){};
\draw[thick] (0,0) -- (0.8,0.9){};
\draw[thick] (0.8,0.9) -- (1.5,-0.5){};
\draw[thick] (0.8,-1.2) -- (1.5,-0.5){};
\draw[thick] (0.8,0.9) -- (4.3,2.4){};
\draw[thick] (4.3,2.4) -- (5,1){};
\draw[thick] (0.8,-1.2) -- (4.3,0.3){};
\draw[thick] (4.3,0.3) -- (5,1){};
\draw[thick] (0,0) -- (0.9 ,0.9* 3/7){};
\draw[thick] (1.2,1.2*3/7) -- (3.5,1.5){};
\draw[thick] (3.5,1.5) -- (4.3,2.4){};
\draw[thick] (3.5,1.5) -- (5,1){};
\draw[thick] (3.5,1.5) -- (3.5 +0.5,1.5 -1.5*0.5){};
\draw[thick] (3.5 +0.7,1.5 -1.5*0.7)--(4.3,0.3){};
\filldraw(0, 0) circle (2pt);
\filldraw(0.8, -1.2) circle (2pt);
\filldraw(0.8, 0.9) circle (2pt);
\filldraw(1.5, -0.5) circle (2pt);
\filldraw(3.5, 1.5) circle (2pt);
\filldraw(4.3, 0.3) circle (2pt);
\filldraw(4.3, 2.4) circle (2pt);
\filldraw(5, 1) circle (2pt);
\draw (-0.3,0) node {1};
\draw (0.8,-1.5) node {7};
\draw (1.6,-0.1) node {2};
\draw (0.8,1.3) node {5};
\draw (4.3,2.8) node {6};
\draw (3.5,1.8) node {4};
\draw (5.3,1) node {3};
\draw (4.4,-0.1) node {8};
\end{tikzpicture}
\end{center}
\end{figure}
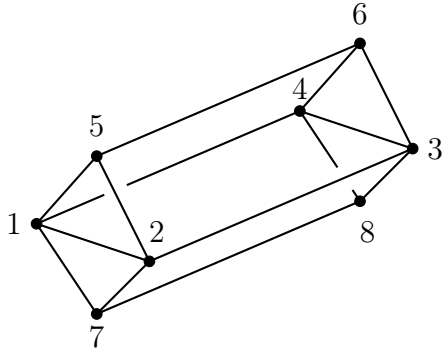

We can apply the sums-of-squares-test to the Hermite matrix $\cH(p)$ here. Unfortunately, the matrix is too complicated to do the computations by hand. When using a numerical sums-of-squares-plugin for matlab, such as Yalmip, the result however indicates that $\cH(p)$ is not a sum of squares. In view of Theorem \ref{thm:sos} this  shows again that no power of $p$ admits a determinantal representation. Note that if some power  $p^r$ has a determinantal representation, then it has one of size $4r$. This was proven by Br\"and\'{e}n or follows more generally from Netzer and Thom, Theorem 2.7 \cite{neth}.

Finally, we can apply the sums-of-squares-test also to small perturbations of Br\"and\'{e}n's polynomial. For example, $p$ can be approximated as closely as desired by real-zero polynomials, which have only simple roots on each line through the origin (in other words, the Hermite matrix is positive definite at each point $a\neq 0$). Such a smoothening procedure is for example describe in Nuij \cite{nu}. Still, Yalmip reports that the Hermite matrix is not a sum of squares, if the approximation is close enough. This is exactly what one expects, since the cone of sums of squares of polynomial matrices is closed, and the Hermite matrix depends continuously on the polynomial.
\end{example}

\section{A general construction method}\label{section:construction}

In this section we are interested in the converse of the above result. Namely, can a sums-of-squares decomposition of $\cH(p)$ be used to produce a definite determinantal representation of $p$ or some multiple? We describe a method to do this, which amounts to only solving a system of linear equations.

Let $p=1+p_1+\cdots +p_d\in\R[\xvar]$ be a real-zero polynomial of degree
$d$. 
Since the matrix
$\cH(p)$ is everywhere positive semidefinite, it can be
expressed as a sum of squares if one allows denominators in
$\R[\xvar]$. This generalization of Artin's solution to Hilbert's 17th
problem was first proved by Gondard and Ribenboim in \cite{gori}. We
need to make a slight adjustment to our situation.

\begin{lemma}\label{lemma:gori}
There exist a matrix polynomial $\cQ\in{\rm Mat}_{k\times
  d}\bigl(\R[\xvar]\bigr)$, for some $k>0$, and a homogeneous non-zero polynomial
$q\in\R[\xvar]$ such that
\[
q^2\cH(p)=\cQ^T\cQ.
\]
\end{lemma}
\begin{proof} By the original result of Gondard and Ribenboim
  \cite{gori} there is some non-zero polynomial $q\in\R[\xvar]$ such
  that $q^2\cH(p)=\cQ^T\cQ$ for some $\cQ\in{\rm Mat}_{k\times
    d}(\R[\xvar])$. We want to make $q$ homogeneous.

  Write $q=q_r + q_{r+1}+ \cdots + q_R$, where each $q_i$ is
  homogeneous of degree $i$, and $q_r\neq 0$, $q_R\neq 0$. Since the
  $i$-th diagonal entry in $\cH(p)$ is homogeneous of degree
  $2(i-1)$, each entry in the $i$-th column of $\cQ$ has homogeneous
  parts of degree between $r+i-1$ and $R+i-1.$ Let $\cQ_{\rm min}$ be
  the matrix one obtains from $\cQ$ by choosing only the homogeneous
  part of degree $r+i-1$ of each entry in each $i$-th column. Put
  $\widetilde{\cQ}=\cQ-\cQ_{\rm min}$ and note that all entries in the
  $i$-th column of $\widetilde{\cQ}$ have non-zero homogeneous parts
  only in degrees at least $r+i$. We now compute $q^2\cH(p)=
  \cQ_{\rm min}^T\cQ_{\rm min} + \cQ_{\rm min}^T\widetilde{\cQ} +
  \widetilde{\cQ}^T\cQ_{\rm min} + \widetilde{\cQ}^T\widetilde{\cQ}$, compare
  degrees on both sides, and find $q_{r}^2\cH(p)=\cQ_{\rm
    min}^T\cQ_{\rm min}$, as desired.
\end{proof}

We will now describe the setup that we are going to use for the rest
of this section. We fix a representation of $q^2\cH(p)=\cQ^T\cQ$
as in Lemma \ref{lemma:gori}. As before, let $\phom=t^d\cdot p(t^{-1}\xvar) =
t^d +p_1t^{d-1}+\cdots +p_d\in\R[\xvar,t]$, and consider the free $\R[\xvar]$-module
\[
A = \R[\xvar,t]/(\phom) \cong \bigoplus_{i=0}^{d-1} \R[\xvar]\cdot t^i
\cong\R[\xvar]^d.
\]
Since $P$ is homogeneous, the standard grading induces  a grading on $A$. We shift this grading by $r$, the degree of $q$, and obtain a grading with $\deg(t^i)=r+i$ for
$i=0,\ldots, d-1$. This turns $A$ into
a graded $\R[\xvar]$-module, where $\R[\xvar]$ is equipped with the standard
grading. Furthermore, we equip $A$ with a symmetric $\R[\xvar]$-bilinear
and $\R[\xvar]$-valued map $\langle \cdot , \cdot \rangle_p$ defined by
\[
\langle f, g\rangle_p := f^T\left( q^2  \cH(p)\right) g,
\]
for $f=(f_1,\ldots,f_d)^T$ and $g=(g_1,\ldots,g_d)^T$ in $A$.

Next, consider the map $\cL_t\colon A\rightarrow A$ given by
multiplication with $t$. This is an $\R[\xvar]$-linear map which we can
compute with respect to our chosen basis: 
\[
\cL_t \colon (f_1,\ldots,f_d)^T\mapsto
  (-p_df_d, f_1-p_{d-1}f_d, \ldots, f_{d-1} -p_1f_d)^T.
\]
Note that $\cL_t$ is of degree $1$ with respect to the grading,
i.e.~$\deg\bigl(\cL_t(f)\bigr)=\deg(f)+1$. 
We identify $\cL_t$ with the matrix that represents it, so that
\[
\cL_t=\left(\footnotesize\begin{array}{cccc}
0 & 0 & 0 & -p_d \\1 & 0 & 0 & -p_{d-1}
    \\0 & \ddots & 0 & \vdots \\0 & \cdots & 1 &
    -p_1
\end{array}\right),
\]
which is exactly the {\it companion matrix} of
$\phom$, viewed as a univariate polynomial in $t$. It is well known and
easy to see that $\phom$ is the characteristic polynomial of $\cL_t$, so that
\[
\det\left(I-\cL_t\right)=p
\]
\begin{lemma}\label{lemma:selfad}
The linear map $\cL_t$ is self-adjoint with respect to $\langle \cdot,\cdot\rangle_p$, i.e. $$\langle \cL_tf,g\rangle_p=\langle f,\cL_tg\rangle_p$$ holds for all $f,g\in A.$
\end{lemma}
\begin{proof} We may divide by $q^2$ on both sides and hence assume
  that $q=1$.  It is enough to show $\langle \cL_te_i,e_j\rangle_p
  =\langle e_i ,\cL_te_j\rangle_p$ for all $i,j$, where $e_i$ is the
  $i$-th unit vector. For $i,j<d$, this follows from the fact that
  $\cH(p)$ is a Hankel matrix. For $i=j=d$, it is clear from
  symmetry. So assume $j<i=d$. We find $$\langle \cL_te_d,e_j\rangle_p =
  -\sum_{i=1}^d p_{d-i+1} e_i\cH(p)e_j= -\sum_{i=1}^d
  p_{d-i+1} N_{i+j-2},$$ where $N_k$ is the $k$-th Newton sum of
  $\phom$. On the other hand, we compute $\langle
  e_d,\cL_te_j\rangle_p = \langle e_d,e_{j+1}\rangle_p = N_{d+j-1}$. In
  conclusion, we have to show that
\[
\sum_{i=0}^{d}p_{d-i}N_{i+j-1}=0,
\]
where we have set $p_0=1$. This statement is equivalent to
$\sum_{i=0}^d p_i N_{k-i}=0$, where $k=d+j-1\geq d$. This last
equation, however, follows immediately from the Newton identity
$kp_k + \sum_{i=0}^{k-1} p_iN_{k-i}=0$, where we let $p_k=0$ for $k>d$. 
\end{proof}

Let $B=\R[\xvar]^k$. The $k\times d$-matrix $\cQ$ in the decomposition of $\cH(p)$
describes an $\R[\xvar]$-linear map $A=\R[\xvar]^d\rightarrow B$,
$f\mapsto \cQ f$. From the degree structure of $\cH(p)$, we see
that each entry in the $i$-th column of $\cQ$ is homogeneous of degree
$r+i-1$. So $\cQ$ is of degree $0$ with respect to the canonical grading
on $B$.

\begin{lemma}\label{lemma:Qproperties}\mbox{}
\begin{enumerate}
\item If $p$ is square-free, then $\cQ\colon A\rightarrow B$ is
  injective.
\item We have 
\[
\langle f,g\rangle_p = \langle \cQ f,\cQ g\rangle
\] 
for all $f,g\in A$. In other words, $\cQ$ is an isometry, taking
$\langle\;,\;\rangle_p$ to the canonical bilinear form
$\langle\;,\;\rangle$ on $B$.
\end{enumerate}
\end{lemma}

\begin{proof}
(2) is immediate from the fact that $q^2\cH(p)=\cQ^T\cQ$. (1)
If $\cQ f=0$, then $$0=\langle \cQ f,\cQ f\rangle= \langle f,f\rangle_p = q^2 \cdot f^T \cH(p)f.$$ For each $a\in\R^n$ for which $p(ta)$ has only distinct roots, the matrix $\cH(p)(a)$ is positive definite. So $f(a)=0$ for generic $a$, and thus $f=0$.
\end{proof}

\noindent Time for a brief summary of what we have done so far.

\begin{setup}\mbox{}
\begin{itemize}
\item
Let $p\in\R[\xvar]$ be a real-zero polynomial of degree $d$ with $p(0)=1$, and let
$\cH(p)$ be its parametrized Hermite matrix. Fix a
decomposition $q^2\cH(p)=\cQ^T\cQ$, where $q$ is homogeneous of
degree $r$ and $\cQ$ is a matrix of size $k\times d$ with entries in $\R[\xvar]$.


\item We have equipped the free module $A=\R[x]^d$ with a particular
  grading and with a bilinear form $\langle\;,\;\rangle_p\colon
  A\times A\rightarrow A$. 

\item Let $B=\R[x]^k$ be equipped with the canonical
  bilinear form and the canonical grading.

\item The map $\cQ\colon A\rightarrow B$ is an isometry and
  of degree $0$.
\item Let $\cL_t$ be the companion matrix of $t^dp(t^{-1}\xvar)$ with
  respect to $t$, so
  that 
\[
\det(I-\cL_t)=p.
\]
The map $\cL_t\colon A\rightarrow A$ is self-adjoint
with respect to $\langle \cdot,\cdot\rangle_p$ and of degree $1$.
\end{itemize}
\end{setup}


\noindent The following is our main result.

\begin{theorem}\label{thm:main}
 Let $p\in\R[\xvar]$ be a square-free real-zero polynomial of
  degree $d$ with $p(0)=1$. Assume that there exists a
  homogeneous symmetric linear matrix polynomial $\cM$ of size $k\times k$ such
  that the following diagram commutes:
\[
\xymatrix{ \R[\xvar]^d=A \ar@{->}^\cQ[r] \ar@{->}^{\cL_t}[d]& B=\R[\xvar]^{k}
  \ar@{->}^{\cM}[d] \\ \R[\xvar]^d= A \ar@{->}^{\cQ}[r] & B=\R[\xvar]^{k}}
\]
Then $p$ divides $\det(I -\cM)$.
\end{theorem}

\begin{remark}
Note that the above described setup exactly means that we can hope for such a linear symmetric $\cM$ to exist. Indeed the "strange" symmetry of $\cL_t$ is transformed into the standard symmetry by $\cQ$, and the "strange" grading is translated to the standard grading.
\end{remark}

\begin{proof}
For generic $a\in \R^n$, the map $\cQ(a)$ is injective by Lemma
\ref{lemma:Qproperties}. Therefore, all eigenvalues of $\cL_t(a)$ are
also eigenvalues of $\cM(a)$. The eigenvalues of $\cL_t(a)$ are precisely
the zeros of $\phom(t,a)$, i.e.~the inverses of the zeros of
$p(ta)$. So $q=\det(I-\cM)$ vanishes on the zero set of $p$, by
Proposition \ref{prop:eigenvalues}. Since $p$ is a square-free real
zero polynomial, the ideal $(p)$ generated by $p$ in $\R[\xvar]$ is real-radical (see
Bochnak, Coste and Roy \cite{bcr}, Theorem 4.5.1(v)). It follows that
$q$ is contained in $(p)$, in other words $p$ divides $q$.
\end{proof}

\begin{remark}
  Whether there exists such $\cM$ can be decided by solving
  a system of linear equations. Indeed, set $\cM=x_1M_1+\cdots
  +x_nM_n,$ where the $M_i$ are symmetric matrices with indeterminate
  entries. The equation $\cM\cQ=\cQ\cL_t$ of matrix polynomials can be
  considered entrywise, and comparison of the coefficients in $\xvar$ gives rise to a system of linear
  equations in the entries of the $M_i$.
\end{remark}

\begin{example}\label{quad2} Let $p\in\R[\xvar]$ be quadratic. Write
  $p=\xvar^T A \xvar +b^T\xvar +1$ with $A\in{\rm Sym}_n(\R)$ and $b\in\R^n$. We
  have seen in Example \ref{quad1} that $\cH(p)$ admits a sums of
  squares decomposition if $p$ is a real-zero polynomial, given by
  the matrix 
\[
\cQ=\sqrt{2}\cdot \left(\footnotesize\begin{array}{cc}1 & -\frac12 b^T\xvar \\0 &
    \frac12 v_1^T\xvar \\\vdots & \vdots \\0 & \frac12 v_n^T\xvar
\end{array}\right)
\]
if $bb^T-4A=\sum_{i=1}^n v_iv_i^T$. It is now easy to find a
homogeneous linear matrix polynomial $\cM$ that makes the diagram in
Theorem~\ref{thm:main} commute, namely we can take
\[
\cM=\frac12\cdot \left(\footnotesize\begin{array}{cccc}-b^T\xvar &  v_1^T\xvar &
    \cdots &  v_n^T\xvar \\v_1^T\xvar & -b^T\xvar & 0 & 0 \\\vdots & 0
    & \ddots & 0 \\v_n^T\xvar & 0 & 0 &  -b^T\xvar\end{array}\right).
\]
The resulting determinantal representation is
\[
\det\left(I-\cM\right)=\biggl(1+\frac12\cdot b^T\xvar\biggr)^{n-1}\cdot p.
\]

To give an explicit example, consider $p=(x_1+\sqrt{2})^2
-x_2^2-x_3^2-x_4^2-x_5^2$, which itself does not admit a determinantal
representation (by Netzer and Thom \cite{neth}).  The procedure just described now gives rise to the
linear matrix polynomial
\[
\cM=\left(\footnotesize\begin{array}{cccccc}-\sqrt{2}x_1 & x_1 & x_2 & x_3 & x_4 &
    x_5 \\x_1 & -\sqrt{2}x_1 & 0 & 0 & 0 & 0 \\x_2 & 0 & -\sqrt{2}x_1
    & 0 & 0 & 0 \\x_3 & 0 & 0 & -\sqrt{2}x_1 & 0 & 0 \\x_4 & 0 & 0 & 0
    & -\sqrt{2}x_1 & 0 \\x_5 & 0 & 0 & 0 & 0 &
    -\sqrt{2}x_1\end{array}\right)
\] 
and
finally 
\[
\det(I-\cM)=(1+\sqrt{2}x_1)^4\cdot p.
\]
\end{example}

\begin{example}
There are also examples where no suitable $\cM$ exists. We are grateful to Rainer Sinn and Cynthia Vinzant for helping us find this example. Consider the
plane cubic $p=(x_1-1)^2(x_1+1)-x_2^2.$ One computes 
\[
\cH(p)=\left(\footnotesize\begin{array}{ccc}3 & x_1 & 3x_1^2+2x_2^2 \\x_1 &
    3x_1^2+2x_2^2 & x_1^3+3x_1x_2^2 \\3x_1^2+2x_2^2 & x_1^3+3x_1x_2^2
    & 3x_1^4+8x_1^2x_2^2+2x_2^4\end{array}\right)=\cQ^T\cQ,
\]
where
\[
\cQ=\left(\footnotesize\begin{array}{ccc}0 & x_2 & ax_1x_2 \\0 & -x_2 &
    bx_1x_2 \\\sqrt{2} & \sqrt{2}x_1 & \sqrt{2}(x_1^2+x_2^2) \\1 &
    -x_1 & x_1^2\end{array}\right)
\]
and $a=\frac12(\sqrt{7}+1)$, $b=\frac12(\sqrt{7}-1)$. The equation
$\cM\cQ=\cQ\cL_t$ has 12 entries, each of which gives rise to several linear
equations by comparing coefficients in $\xvar.$
One can check that already the equations obtained from the first two rows of
$\cM\cQ=\cQ\cL_t$ are unsolvable. 
\end{example}

\section{Rational representations of degree one} 
There is always a way to make the diagram from the last section commute, if one allows for \emph{rational}  linear matrix polynomials. This will lead to rational determinantal representations, as described now.

Let $p$ be a square-free real-zero polynomial. Since the
parametrized Hermite matrix $\cH(p)$ evaluated at a point $a\in\R^n$
is positive definite for generic $a$, the matrix polynomial $\cH(p)$
is invertible over the function field $\R(\xvar)$. Recall that the
degree of a rational function $f/g\in\R(\xvar)$ is defined as
$\deg(f)-\deg(g)$. Furthermore, we say that $f/g$ is homogeneous if
both $f$ and $g$ are homogeneous, not necessarily of the same
degree. Equivalently, $f/g$ is homogeneous of degree $d$ if and only
if $(f/g)(\lambda a)=\lambda^d(f/g)(a)$ holds for all $a\in\R^n$ with
$g(a)\neq 0$.

\begin{theorem}\label{thm:rat} Let $p$ be a square-free real-zero polynomial. Write $q^2\cH(p)=\cQ^T\cQ$ with $q$ homogeneous as in Lemma \ref{lemma:gori} and let 
\[
\cM:=q^{-2}\cQ\cL_t\cH(p)^{-1}\cQ^T.
\]
The matrix $\cM$ is symmetric with entries in $\R(x)$ homogeneous of degree $1$,
and satisfies
\[
\det(I-\cM)=p.
\]
\end{theorem}

\begin{proof}
Abbreviate $\cH(p)$ by $\cH$ and $\cL_t$ by $\cL$.
By Sylvesters determinant theorem, we have
$\det(I_k-\mathcal{A}\mathcal{B})=\det(I_d-\mathcal{B}\mathcal{A})$
for any matrix polynomials $\mathcal{A}$ of size $k\times d$ and
$\mathcal{B}$ of size $d\times k$. In our situation, this yields
\[
\det (I_k-\cM)= \det(I_k
-q^{-2}\cQ\cL\cH^{-1}\cQ^T)= \det(I_d - q^{-2}\cL\cH^{-1}\cQ^T\cQ) =
\det(I_d-\cL)=p.
\]
 We find 
\[
\cM^T= q^{-2}\cQ (\cH^{-1})^T\cL^T\cQ^T= q^{-2}\cQ \cL\cH^{-1}\cQ^T=\cM,
\]
where we have used $\cL^T\cH=\cH^T\cL$, which is Lemma \ref{lemma:selfad}. Thus $\cM$ is symmetric.\\
Let $r$ be the degree of $q$. By examining the degree structure of
$q^2\cH$, we find
\begin{align*}
  \cQ(\lambda a)&=\cQ(a)\cdot{\rm diag}(\lambda^r,
  \lambda^{r+1},\ldots,\lambda^{r+d-1})\\
\cH(\lambda a) &= {\rm
    diag}(\lambda^0,\ldots,\lambda^{d-1})\cdot\cH(a)\cdot{\rm
    diag}(\lambda^0,\ldots,\lambda^{d-1})\\
\cL(\lambda a) &={\rm
    diag}(\lambda^d,\ldots,\lambda^1)\cdot\cL(a)\cdot{\rm
    diag}(\lambda^{-d+1},\lambda^{-d+2}\ldots,\lambda^{0})
\end{align*}
for all $a\in \R^n$ and $\lambda\neq 0$. Hence for all $a\in\R^n$ for
which $\cH(a)$ is invertible and $q(a)\neq 0$, and all $\lambda\neq 0$, we have 
\begin{align*} \cM(\lambda a) &=\lambda^{-2r} q(a)^{-2} \cQ(a)\cdot
  {\rm diag}(\lambda^r,\ldots,\lambda^{r+d-1})\cdot {\rm
    diag}(\lambda^d,\ldots,\lambda)\cL(a)\cdot{\rm
    diag}(\lambda^{-d+1},\ldots,\lambda^0)\\
  &\quad \cdot{\rm diag}(\lambda^{0},\ldots,\lambda^{-d+1})\cdot
  \cH^{-1}(a)\cdot{\rm diag}(\lambda^0,\ldots,\lambda^{-d+1})\cdot{\rm
    diag}(\lambda^r,\ldots,\lambda^{r+d-1})\cdot\cQ^T(a)\\
  &=\lambda^{-2r}q(a)^{-2}\cQ(a)\lambda^{r+d}\cL(a)\lambda^{-d+1}\cH^{-1}(a)\lambda^r\cQ^T(a)\\
  &= \lambda\cdot\cM(a).
\end{align*}
\end{proof}

\begin{remark}
Note that a representation $p=\det\left(I-\cM\right)$ as in Theorem \ref{thm:rat} gives an algebraic certificate for $p$ being a real-zero polynomial. Since $p(ta)=\det\left( I-t\cM(a)\right),$ using homogeneity, the zeros of $p(ta)$ are just the inverses of the eigenvalues of $\cM(a)$. Since $\cM$ is symmetric, all of these zeros are real. Theorem \ref{thm:rat} now states that such an algebraic certificate exists for \emph{each} real-zero polynomial $p$.
\end{remark}


\begin{example}
Consider the quadratic polynomial $p=(x_1+1)^2 -x_2^2-x_3^2 -x_4^2.$
We have 
\[
\cH=\left(\!\!\!\footnotesize\begin{array}{cc}2 & -2x_1 \\-2x_1 & 2(x_1^2+x_2^2+x_3^2
    +x_4^2)\end{array}\!\!\!\right)=\cQ^T\cQ
\quad\text{with}\quad
\cQ^T  = \left(\!\!\!\footnotesize\begin{array}{cccc}\sqrt{2} & 0 & 0 & 0 \\
  -\sqrt{2}x_1 & \sqrt{2}x_2 & \sqrt{2}x_3 & \sqrt{2}x_4
\end{array}\!\!\!\right),
\]
which results in 
\[
\cM= \left(\!\!\footnotesize\begin{array}{cccc}-x_1 & x_2 & x_3 & x_4
    \\x_2 & -\frac{x_1x_2^2}{x_2^2+x_3^2+x_4^2} &
    -\frac{x_1x_2x_3}{x_2^2+x_3^2+x_4^2} &
    -\frac{x_1x_2x_4}{x_2^2+x_3^2+x_4^2} \\x_3 &
    -\frac{x_1x_2x_3}{x_2^2+x_3^2+x_4^2} &
    -\frac{x_1x_3^2}{x_2^2+x_3^2+x_4^2} &
    -\frac{x_1x_3x_4}{x_2^2+x_3^2+x_4^2} \\x_4 &
    -\frac{x_1x_2x_4}{x_2^2+x_3^2+x_4^2} &
    -\frac{x_1x_3x_4}{x_2^2+x_3^2+x_4^2} &
    -\frac{x_1x_4^2}{x_2^2+x_3^2+x_4^2}\end{array}\!\!\right).
\]
\end{example}

\end{document}